\documentclass[12pt, reqno]{amsart}

\usepackage{amssymb,latexsym,amsmath,amsfonts}
\usepackage{mathrsfs}
\usepackage{graphicx}
\usepackage[usenames]{color}
\usepackage{hyperref}
\usepackage{comment}
\usepackage{enumitem}
\usepackage{titlesec}

\titleformat{\section}
  {\Large\bfseries}
  {\thesection}{0.8em}{}

\titleformat{\subsection}
  {\large\bfseries}
  {\thesubsection}{0.8em}{}

\titleformat{\subsubsection}
  {\normalsize\bfseries}
  {\thesubsubsection}{0.8em}{}
\definecolor{DPurple}{rgb}{0.46,0.2,0.69}

\numberwithin{equation}{section}

\allowdisplaybreaks

\theoremstyle{definition}
\newtheorem{definition}{Definition}[section]

\theoremstyle{remark}

 \theoremstyle{plain}
\newtheorem{theorem}[definition]{Theorem}

\newtheorem{lemma}[definition]{Lemma}
\newtheorem{proposition}[definition]{Proposition}
\newtheorem{example}{Example}[section]
\newtheorem{corollary}[definition]{Corollary}


\title{Primitive central idempotents of a finite group over a field}

\author{Satyam Shukla}
\address{Department of Mathematics, Babasaheb Bhimrao Ambedkar Univeristy, Lucknow-226025, India}
\email{satyam7112shukla@gmail.com}
\author{Jagmohan Tanti}
\address{Department of Mathematics, Babasaheb Bhimrao Ambedkar Univeristy, Lucknow-226025, India}
\email{jagmohan.t@gmail.com} 
\begin{document}
\begin{abstract}
We give a method to compute the primitive central idempotents of a semi-simple finite group algebra associated with an irreducible character $\chi$ such that for every $g\in G$, either $\chi(g)=0$ or all eigenvalues of $\rho(g)$ have the same order. We then show several consequences which highlight the significance and possible applications of our approach.
\end{abstract}
\maketitle

\section{Introduction} 
In this paper, $\mathbb{F}=\mathbb{F}_q$ denotes a finite field with $q$ elements, and $G$ is a finite group such that the group algebra $\mathbb{F}[G]$ is semi-simple, or equivalently $(q,|G|)=1$. The study of Wedderburn decomposition of $\mathbb{F}[G]$ is a tool to deal with several classical problems(see, e.g., \cite{Broch}). The problem of determining the Wedderburn decomposition of $\mathbb{F}[G]$ naturally leads to computation of the primitive central idempotents of $\mathbb{F}[G]$. For an irreducible character $\chi$ of $G$, let $\mathbb{F}(\chi)$ denote the character field of $\chi$. All the characters of any finite group are considered as characters in $\mathbb{F}(\chi)$, $e(\chi)=\frac{\chi(1)}{|G|}\sum_{g\in G}\chi(g^{-1})g$ is the primitive central idempotents of $\mathbb{F}(\chi)[G]$ associated to $\chi$ and $e_{\mathbb{F}}(\chi)$ is the only primitive central idempotent $e$ of $\mathbb{F}[G]$ such that $e(\chi)e\neq 0$. The Galois group $Gal(\mathbb{F}(\chi)/\mathbb{F})$ of the field extension $\mathbb{F}(\chi)/\mathbb{F}$ acts on $\mathbb{F}(\chi)[G]$ by acting on the coefficients, that is for $\sigma\in Gal(\mathbb{F}(\chi)/\mathbb{F})$, we have
\begin{center}
  $$\sigma\sum_{g\in G}a_gg=\sum_{g\in G}\sigma(a_g)g.$$
\end{center}
We recall the following formula \cite{T. Yamada}
\begin{center}
$$e_{\mathbb{F}}(\chi)=\sum_{\sigma\in Gal(\mathbb{F}(\chi)/\mathbb{F})}\sigma(e(\chi)).$$  
\end{center}
The decomposition of $\mathbb{F}[G]$ as a direct sum of matrix rings over division rings enables to produce all the ideals of $\mathbb{F}[G]$. If $e_1, ..., e_m$ are the primitive central idempotents of $\mathbb{F}[G]$, then $\mathbb{F}[G]=\mathbb{F}[G]e_1\oplus ...\oplus \mathbb{F}[G]e_m$ is the Wedderburn decomposition of $\mathbb{F}[G]$. If $G$ is cyclic, then the primitive central idempotents of $\mathbb{F}[G]$ are in a one-to-one correspondence with $q$-cyclotomic classes module \cite{V. S. Pless}, and using this, it is not difficult to compute the primitive idempotents and the Wedderburn decomposition of any commutative finite group algebra [\ref{P: abelian}].\smallskip
 
The explicit description of the primitive central idempotents of the rational group algebra of a finite abelian group $G$ is well known, and has been given in several ways (\cite{Ayoub} \cite{Goodaire} \cite{Jespers} \cite{Olivieri}). A description of the primitive central idempotents of $\mathbb{Q}[G]$, in the case when $G$ is a finite nilpotent group, has been given by Jespers, Leal, and Paques \cite{Jespers}. Olivieri Rio, and Simon \cite{Olivieri} have obtained an expression of the primitive central idempotents of $\mathbb{Q}[G]$ associated with a monomial irreducible characters of $G$.\smallskip

G. K. Bakshi, and I. B. S. Passi \cite{Bakshi}  established the primitive central idempotents of the rational group algebra associated with complex irreducible characters satisfying a certain property, here called the property $\wp$. A complex irreducible character $\chi$ of a finite group $G$, with representation $\rho$, has the property $\wp$ if, for every $g\in G$, either $\chi(g)=0$ or all eigenvalues of $\rho(g)$ have the same order.\smallskip

Our primary goal is to calculate the primitive central idempotents of the finite group algebra linked to irreducible characters which display certain property, denoted as $\varrho$. An irreducible characters $\chi$ of a finite group $G$, associated with the representation $\rho$, is considered to possess property $\varrho$ if, for every $g\in G$, either $\chi(g)=0$ or all eigenvalues of $\rho(g)$ have the same order.\smallskip

In Section 3, We give the main result (Theorem \ref{T: complex}), a method to compute the primitive central idempotents of $\mathbb{F}[G]$ associated with irreducible characters of $G$ having property $\varrho$ and we then show several consequences which highlight the significant and possible applications of our approach. We show that any irreducible character $\chi$ of a group $G$ of degree $\sqrt{[G:Z(\chi)]}$, where $Z(\chi)/\text{ker}(\chi)$ is the centre of $G/\text{ker}(\chi)$, has the property $\varrho$, and we obtain explicit expression of primitive central idempotent of $\mathbb{F}[G]$ associated with such a character $\chi$ (corollary 1). We thus derive primitive central idempotents of $\mathbb{F}[G]$ associated with any irreducible character $\chi$ of degree $\sqrt{[G:Z(G)]}$, where $Z(G)$ is the centre of $G$ (corollary 2). The expressions of primitive central idempotents of $\mathbb{F}[G]$ associated with any irreducible character $\chi$ with $G/Z(\chi)$ abelian are also obtained (corollary 3).\smallskip

In Section 4, we apply the results of section 3 to obtain primitive central idempotent in the group algebras of certain finite groups.

\subsection*{Notations}
In this article, $G$ is a finite group, and $\text{o}(G)$ denotes the order of $G$. By $H\leq G$ (resp. $H\triangleleft G$) means that $H$ is a subgroup(resp. normal subgroup) of $G$. If $H\leq G$ then $N_G(H)$ denotes the normalizer of $H$ in $G$ and we set $\hat{H}=\frac{1}{\text{o}(H)}\sum_{h\in H}h$, an idempotent of $\mathbb{F}[G]$. For $g\in G$, $\langle g\rangle$ denotes the subgroup of $G$ generated by $g$ and $\text{o}(g)$ denotes the order of $g$. For any $H\leq G$, $[G:H]$ denotes the index of $H$ in $G$. Following Osnel Broch et al. [3], if $G$ is cyclic then the set $G^*$ of irreducible characters of $G$ is a group with the natural product $(\chi_1\chi_2)(g)=\chi_1(g)\chi_2(g)$, for $\chi_1,\chi_2\in G^*$ and $g\in G$. Furthermore $G$ and $G^*$ are isomorphic and in particular $G^*$ are precisely the faithful representations of $G$ and let $C(G)=C_q(G)$ denotes the set of $q$-cyclotomic classes of $G^*$ that contains generators of $G^*$. Let $N\trianglelefteq G$ such that $G/N$ is cyclic of order $k$ and $C\in C(G/N)$. If $\chi\in C$ and $tr=tr_{\mathbb{F}(\zeta_k)/\mathbb{F}}$ denotes the trace of field extension $\mathbb{F}(\zeta_k)/\mathbb{F}$, then we set \\\\
$\epsilon_C(G,N)=\frac{1}{\text{o}(G)}\sum_{g\in G}tr(\chi(\overline{g}))g^{-1}=[G:N]^{-1}\hat{N}\sum_{X\in G/N}tr(\chi(X))g^{-1}_X$,\\\\ where $\bar{g}$ denotes the image of $g$ in $G/N$ and $g_X$ denotes a representative of $X\in G/N$ and $\epsilon_C(G,G)=\hat{G}$. For any complex character $\chi$ of $G$, $\text{ker}(\chi)=\{g\in G|\chi(g)=\chi(1)\}$. By $Irr(G)$, we denote the set of irreducible complex characters of $G$. For any characters $\chi$ and $\psi$ of $G$,
\begin{center}
  $[\chi,\psi]=\frac{1}{\text{o}(G)}\sum_{g\in G}\chi(g)\overline{\psi(g)}$
\end{center}
is the inner product of the character $\chi$ and $\psi$.\smallskip

Let $N\trianglelefteq G$, and $H/N=\langle Na \rangle$ a cyclic subgroup of $G/N$. If $H\neq N$ and $\text{o}(H/N)=p^{\alpha_1}_{1}p^{\alpha_2}_{2}...p^{\alpha_n}_{n}$, $p_i's$ distinct primes, $r_s's\geq 1$, we define
\begin{center}
  $E_{N,H}=\frac{(p_1-1)(p_2-1)...(p_n-1)}{\text{o}(N)p_1p_2...p_n}\sum_{g\in K}\frac{\mu(d(g))}{\phi(d(g))}g$,
\end{center}
where, for any $g\in K$, $d(g)$ is order of $g$ modulo $N$, and $K/N$ is the subgroup of $H/N$ of order $p_1p_2...p_n$. Set $E_{N,N}=\hat{N}$.
\section{Preliminaries}
This section will provide more rigorous definitions and some lemmas. We also introduce some auxiliary results which will be used repeatedly.\smallskip

Recall that for any integer $n\geq1$,\smallskip

$\phi(n):$= number of integers $i$, $1\leq i \leq n$, $gcd(i,n)$,

\[
\mu(n) :=
\begin{cases}
1, & \text{if } n = 1,\\
(-1)^r, & \text{if $n$ is square-free and is a product of $r$ distinct primes},\\
0, & \text{if $n$ is not square-free}.
\end{cases}
 \]
\begin{lemma}\cite[Lemma 1]{Bakshi} \label{L: root} For any integer $n\geq1$ and any primitive $n$th root of unity $\zeta$,
\begin{equation}\nonumber
  \sum_{1 \leq i \leq n, (i,n)=1}\zeta^{i} =\mu(n). 
\end{equation}
\end{lemma}
\begin{proof} 
We first consider the case when $n=p^k, p$ prime, $k\geq1$. The result is clear if $k=1$. Suppose $k\geq2$. Let $S= \{a+b p|1\leq a\leq p-1, 0\leq b\leq p^{k-1}-1\}$. Then $S$ is reduced residue system modulo $p^k$. Therefore, 
 \begin{eqnarray}\nonumber
    \sum_{1 \leq i \leq n, (i,n)=1}\zeta^{i}&=& \sum_{i\in S}\zeta^{i} \\\nonumber
    &=& \bigg(\sum_{a}\zeta^{a}\bigg)\bigg(\sum_{b}(\zeta^b)^{i}\bigg) \\\nonumber
   &=& \bigg(\sum_{a}\zeta^{a}\bigg)\bigg(\frac{\zeta^{p^{k}}-1}{\zeta^p-1}\bigg) \\\nonumber
    &=& 0 \\\nonumber
    &=& \mu(p^k)\\\nonumber
 \end{eqnarray}
 Now let $n=p^{a_1}_1p^{a_2}_2...p^{a_r}_r, p_j$ distinct primes, $a_j\geq1$ for all $j$. Let $m_j=\dfrac{n}{{p_j}^{a_j}}, 1\leq j\leq r$. For each $j, 1\leq j\leq r$, choose an integer $x_j$ such that $m_jx_j\equiv 1 \mod (p^{a_j}_j)$, and consider the set $S= \{\sum_{j=1}^{r}m_jx_ju_j|1\leq u_j\leq p^{a_j}_j, gcd(u_j, {p_j}^{a_j})=1\}$. Then $S$ is a reduced system modulo $n$ and thus            
 \begin{center}
 $\sum_{1 \leq i \leq n, (i,n)=1}\zeta^{i}=\sum_{i\in S}\zeta^{i}=\bigg(\sum_{u_1}(\zeta^{m_1x_1})^{u_1}\bigg)...\bigg(\sum_{u_r}(\zeta^{m_rx_r})^{u_r}\bigg)$.
 \end{center}
Note that, for each $j, \zeta^{m_jx_j}$ is a primitive $p^{a_j}_jth$ root of unity. Therefore, the required result now follows by using the preceding case.
\end{proof}\smallskip

A ring $R$ is called clean if every element of $R$ is the sum of a unit and idempotent. A ring $R$ is called a $*$-ring (or ring with involution $*$) if there exists an operation $*:R\rightarrow R$ such that\smallskip

$(x+y)^*=x^*+y^*$, $(xy)^*=y^*x^*$, and $(x^*)^*=x$, for all $x,y\in R$.\smallskip

We call an element $p$ of a $*$-ring $R$ a projection if $p$ is a $*$-invariant idempotent, i.e. $p^* = p = p^2$, and a $*$-ring $R$ a $*$-clean ring if each element of $R$ is the sum of a unit and projection. Let $G$ be a finite group and $\mathbb{F}$ a finite field with characteristic not dividing the order of $G$. Consider the group algebra $\mathbb{F}[G]$ equipped with a involution $*$ (The canonical involution defined by $*:\mathbb{F}[G]\rightarrow \mathbb{F}[G]$ such that $(\sum_{g\in G}a_gg)^*\rightarrow\sum_{g\in G}a_gg^{-1}$)\cite{Wood} 
\begin{lemma}{(Maschke)}\label{L: characteristic}
 \cite[Theorem 1.2]{S. Lang}. Let $G$ be a finite group and $\mathbb{F}$ a field whose characteristic does not divide the order of $G$. Then the group ring $\mathbb{K}[G]$ is semi-simple. 
\end{lemma}
\begin{proposition}\label{P: abelian}
\cite{Broch}. If $G$ is a finite abelian group of order $n$ and $\mathbb{F}$ is a finite field of order $q$ such that $gcd(n,q)=1$, then the map $(N,C)\rightarrow\epsilon_C(G,N)$ is bijection from the set of pair $(N,C)$ with $N\trianglelefteq G$, such that $G/N$ is cyclic and $C\in C_q(G/N)$ to the set of primitive central idempotents of $\mathbb{F}[G]$ i.e. $e_{\mathbb{F}}(\chi)=\frac{1}{\text{o}(G)}\sum_{g\in G}tr(\psi(\bar{g}))g^{-1}=\epsilon_C(G,N)$. Further for every $N\triangleleft G$ and $C\in C(G/N)$, $\mathbb{F}[G]\epsilon_C(G/N)\cong \mathbb{F}(\zeta_k)$, where $\text{k}=[G:N]$ and $N=\text{ker}(\chi)$ and let $\psi$ be the faithful character of $G/N$ given by $\psi(\bar{g})=\chi(g)$.
\end{proposition}
\begin{lemma}\label{L: commutative}
 \cite[Theorem 2.2]{C. Li}. A commutative $*$-ring is $*$-clean if and only if it is clean and every idempotent is self-adjoint.
\end{lemma}
 \section{Some Results}
Let $\chi\in Irr(G)$ and $\rho$ a representation of $G$ affording the character $\chi$. Let $\phi$ denotes the corresponding representation of $G/\text{ker}(\chi)$ and $\psi$ the character of $\phi$. The following result provides some equivalent conditions for $\chi \in Irr(G)$ to have the property $\varrho$.
\begin{proposition}
For an irreducible character $\chi$ of finite group $G$, and any representation $\rho$ affording $\chi$, the following conditions on an element $g\in G$ are equivalent:\smallskip

 (\romannumeral1)  The associated character $\phi$ is expressible as a sum of faithful irreducible characters of the factor group $G/\text{ker}(\chi)$;
 
 (\romannumeral2) The matrix $\rho(g)$ has all its eigenvalues of the same order;
 
 (\romannumeral3) The mapping $\phi$ annihilates every primitive central idempotents of the group algebra $\mathbb{F}[G/\text{ker}(\chi)]$.
\end{proposition} 
 \begin{proof}
 It is enough to prove the proposition when $\text{ker}(\chi)=1$, and we thus assume this to be the case.\smallskip
 
 {\it(\romannumeral1)}$\implies${\it(\romannumeral2)}: Suppose (\romannumeral1) holds. Let $\zeta_1, \zeta_2,..., \zeta_{\chi(1)}$ be all the eigen-values of $\rho(g)$. Let $t= min(\text{o}(\zeta_1), \text{o}(\zeta_2),...,\text{o}(\zeta_{\chi(1)}))$. Suppose (\romannumeral2) does not hold. Then $t<$ \text{o}($g$). After renumbering, if necessary, suppose \text{o}($\zeta_1$) = $t$. Let $\gamma: \langle g \rangle \rightarrow \mathbb{F}(p\equiv 1(\text{mod n}))$ be the linear character of $\langle g \rangle$ defined by $g\rightarrow\zeta_1$, $\gamma$ is not a faithful character of $\langle g \rangle$ and
 \begin{eqnarray}\nonumber
   [\chi|_{\langle g \rangle},\gamma] &=&  \frac{1}{\text{o}(g)}\sum_{h\in {\langle g \rangle}}\chi(h)\overline{\gamma(h)} \\\nonumber
    &=& \frac{1}{\text{o}(g)}\sum_{h\in {\langle g \rangle}}^{\text{o}(g)-1}(\zeta^i_1+\zeta^i_2+...+\zeta^i_{\chi(1)})\zeta^{-i}_1  \\\nonumber
    &=& \text{the multiplicity of the eigen value $\zeta_1$}\\\nonumber
    &\neq& 0,\nonumber
 \end{eqnarray}
 which shows that $\gamma$ is an irreducible constituent of $\chi|_{\langle g \rangle}$, and we have a contradiction to the hypothesis $(\romannumeral1)$.\smallskip
  
 {\it(\romannumeral2)}$\implies${\it(\romannumeral3)}: Now suppose (\romannumeral2) holds. Then there exists an invertible matrix  $Q\in M_{\chi(1)}(\mathbb{C})$ such that $Q^{-1}\rho(g)Q = diag(\zeta_1, \zeta_2,...,\zeta_{\chi(1)})$, where $\text{o}(\zeta_1) = \text{o}(\zeta_2) = ...= \text{o}(\zeta_{\chi(1)}) = k$, say. We first note that k is equal to the order of g. Next we see that for any non-identity subgroup $H = \langle g^l\rangle$ of $\langle g \rangle$,\\
 \begin{eqnarray}\nonumber
   Q^{-1}\rho(\hat{H})Q &=& \frac{1}{\text{o}(g^l)}\sum_{i=1}^{\text{o}(g^l)-1}diag((\zeta^l_1)^i, (\zeta^l_2)^i, ..., (\zeta^l_{\chi(1)})^i)   \\\nonumber
   &=&\frac{1}{\text{o}(g^l)}diag\bigg(\sum_{i=1}^{\text{o}(g^l)-1}(\zeta^l_1)^i, \sum_{i=1}^{\text{o}(g^l)-1}(\zeta^l_2)^i, ..., \sum_{i=1}^{\text{o}(g^l)-1} (\zeta^l_{\chi(1)})^i\bigg) \\\nonumber
    &=& 0.\nonumber
 \end{eqnarray}
 The primitive central idempotents of $\mathbb{F}[\langle g \rangle]$ are given by $\epsilon_C(G,G)=\hat{G}$, and\smallskip
 
 $\epsilon_C(G, H) = \frac{1}{|G|}\sum_{g\in G}tr(\chi(gH))g^{-1} = [G:H]^{-1}\hat{H}\sum_{X\in G/H}tr(\chi(X))g^{-1}_X,$
 
 where $g_X$ denote a representative of $X\in G/H$. Since $\rho(\hat{H})=0$ for any non-identity subgroup H of $\langle g \rangle$, we have $\rho(\epsilon_C(\langle g \rangle, \langle g \rangle))=0$ and for $l|\text{o}(g), l\neq 1, \text{o}(g),$ $\rho(\epsilon_C(G, H)=0$\smallskip
 
 and consequently (\romannumeral3) holds.\smallskip
 
 {\it(\romannumeral3)}$\implies${\it(\romannumeral1)}: Finally, suppose (\romannumeral3) holds. We can write $\chi_{|_{\langle g \rangle}}=\sum_{\gamma\in Irr(\langle g \rangle)}n_\gamma\gamma,    n_\gamma\in \mathbb{Z}, n_\gamma\geq0$. we need to prove that $n_\gamma=0$, if $\gamma$ is not a faithful character of $\langle g\rangle$. Let $\gamma_0\in Irr(\langle g\rangle)$ be not a faithful character, then $Ker(\gamma_0)=\langle g^l\rangle, l\neq0$, and therefore, by (\romannumeral3), $\rho$ maps $e=\epsilon_C(\langle g\rangle, \langle g^l\rangle)$ to zero. Hence
  \begin{eqnarray}\nonumber
    0 &=& \chi_{|_{\langle g\rangle}}(e)=\sum_{\gamma\in Irr(\langle g\rangle)}n_\gamma\gamma(e).
  \end{eqnarray}
 But for $\gamma\in Irr(\langle g\rangle)$,
 \begin{center}
 \[
\gamma(e) =
\begin{cases}
1, & \text{if $ker(\gamma)=\langle g^l\rangle$ },\\
0, & \text{otherwise}.
\end{cases}
 \]
 \end{center} 
 Therefore, the above equation yields
 \begin{eqnarray}\nonumber
   \sum_{\gamma\in Irr(\langle g\rangle), \text{ker}(\gamma)=\langle g^l\rangle}n_\gamma &=&0
 \end{eqnarray}
Since $n_\gamma\geq0$ for all $\gamma$, we get that $n_\gamma=0$ for all $\gamma\in Irr(\langle g\rangle)$ with $\text{ker}(\gamma)=\langle g^l\rangle$. In particular, $n_{\gamma_0}=0$, which proves (\romannumeral1).
\end{proof} 

\begin{lemma}\label{L: order}
For a finite group $G$ of order $n$, if a character $\chi\in Irr(G)$ and an element $g\in G$ are such that all eigenvalues of $\rho(g)$(in the representation affording $\chi$) are roots of unity of the same order $d$, then
    \begin{center}
       $\sum_{\sigma \in Gal(\frac{\mathbb{F}(\zeta)}{\mathbb{F}})} \sigma(\chi(g)) = \mu(d)\chi(1)\frac{\phi(n)}{\phi(d)},$
    \end{center}
where $d$ is the order of $g$ modulo $ker(\chi)$.
\end{lemma}
\begin{proof} 
Since $\rho(g)$ is similar to $diag(\zeta_1, \zeta_2,...,\zeta_{\chi(1)})$, where $\text{o}(\zeta_1)=\text{o}(\zeta_2) =...=  \text{o}(\zeta_{\chi(1)}) = d$. It is easy to see that $d$ is actually the order of $g$ modulo $Ker(\chi)$.\\
If $d=1$, then $\chi(g)=\chi(1)$ and, therefore,
\begin{center}
  $\sum_{\sigma \in Gal(\frac{\mathbb{F}(\zeta)}{\mathbb{F}})} \sigma(\chi(g))=\chi(1)[\mathbb{F}(\zeta):\mathbb{F}]=\chi(1)\phi(n)$.
\end{center}
If $d>1$, then
\begin{eqnarray}\nonumber
  \sum_{\sigma \in Gal(\frac{\mathbb{F}(\zeta)}{\mathbb{F}})} \sigma(\chi(g)) &=& \sum_{\sigma \in Gal(\frac{\mathbb{F}(\zeta)}{\mathbb{F}})}\sigma(\zeta_1 + \zeta_2 + ...+ \zeta_{\chi(1)}) \\\nonumber
   &=& \sum_{i=1}^{\chi(1)}\sum_{\sigma \in Gal(\frac{\mathbb{F}(\zeta)}{\mathbb{F}})} \sigma(\zeta_i) \\\nonumber
   &=& \sum_{i=1}^{\chi(1)}[\mathbb{F}(\zeta):\mathbb{F}(\zeta^{\frac{n}{d}})]\sum_{\sigma \in Gal(\frac{\mathbb{F}(\zeta^{\frac{n}{d}})}{\mathbb{F}})} \sigma(\zeta_i) \\\nonumber
   &=&\sum_{i=1}^{\chi(1)}[\mathbb{F}(\zeta):\mathbb{F}(\zeta^{\frac{n}{d}})]\sum_{1\leq d,(k,d)=1}\zeta^{k}_i  \\\nonumber
   &=& \frac{\phi(n)}{\phi(d)}\sum_{i=1}^{\chi(1)}\sum_{1\leq d,(k,d)=1}\zeta^{k}_i \\\nonumber
   &=& \mu(d)\chi(1)\frac{\phi(n)}{\phi(d)}   \hfill (\text{Using Lemma \ref{L: root}})\nonumber.
\end{eqnarray}
\end{proof}
Now we are ready to state our main results.
\subsection{Main Result}
In this section, we will prove the main result.
\begin{theorem}\label{T: complex}
If $G$ is a finite group and $\chi$ a complex irreducible character of $G$ satisfying property $\varrho$, then the primitive central idempotent $e_{\mathbb{F}}$ of the group algebra $\mathbb{F}[G]$ attached to $\chi$ can be expressed as a weighted sum of group elements:
  \begin{center}
    $e_{\mathbb{F}}= \frac{1}{\sum_{g\in G, \chi(g)\neq 0}(\frac{\mu(d(g))}{\phi(d(g))})^{2}}\sum_{g\in G, \chi(g)\neq 0}\frac{\mu(d(g))}{\phi(d(g))}g,$  
  \end{center}
  where, for $g\in G$, $d(g)$ denotes order of $g$ modulo \text{ker}($\chi$), and $\mu$ and $\phi$ denote, the Mobius mu and the Euler phi functions respectively.
  \end{theorem}
\begin{proof}
Let $\zeta$ be a primitive $\text{o}(G)th$ root of unity. For computing the primitive central idempotents of $\mathbb{F}[G]$ associated with the character $\chi$ afforded by the representation $\rho(g)$, we have
\begin{eqnarray}\nonumber
  e_{\mathbb{F}}(\chi) &=& \sum_{\sigma \in Gal(\frac{\mathbb{F}(\chi)}{\mathbb{F}})}\sigma \circ e(\chi) \\\nonumber
   &=& \frac{\chi(1)}{|G|}\sum_{\sigma \in Gal(\frac{\mathbb{F}(\chi)}{\mathbb{F}})}\sigma\bigg(\sum_{g\in G}\chi(g)g^{-1}\bigg) \\\nonumber
   &=&  \frac{\chi(1)}{|G|}\sum_{\sigma \in Gal(\frac{\mathbb{F}(\chi)}{\mathbb{F}})}\sum_{g\in G}\sigma(\chi(g))g^{-1} \\\nonumber
   &=&\frac{\chi(1)}{|G|}\sum_{g\in G}\bigg(\sum_{\sigma \in Gal(\frac{\mathbb{F}(\chi)}{\mathbb{F}})}\sigma(\chi(g))\bigg)g^{-1}  \\\nonumber
   &=&\frac{\chi(1)}{|G|}\sum_{g\in G}\frac{1}{[\mathbb{F}(\zeta):\mathbb{F}(\chi)]}\bigg(\sum_{\sigma \in Gal(\frac{\mathbb{F}(\zeta)}{\mathbb{F}})}\sigma(\chi(g))\bigg)g^{-1}  \\\nonumber
   &=& \frac{\chi(1)}{|G|}\frac{1}{[\mathbb{F}(\zeta):\mathbb{F}(\chi)]}\sum_{g\in G}\mu(d(g))\chi(1)\frac{\phi(n)}{\phi(d(g))}g^{-1}\ \ (\text{Using Lemma \ref{L: order}}) \\
   &=&  \frac{(\chi(1))^2}{|G|}\frac{\phi(n)}{[\mathbb{F}(\zeta):\mathbb{F}(\chi)]}\sum_{g\in G}\frac{\mu(d(g))}{\phi(d(g))}g
\end{eqnarray}
as $\chi(g)\neq 0\Leftrightarrow \chi(g^{-1})\neq 0$ and $d(g) = d(g^{-1})$.\\\\
Since $(e_{\mathbb{F}}(\chi))^2=e_{\mathbb{F}}(\chi)$, comparing the coefficient of $1$ on both sides of  
\begin{center}
  $\bigg(\frac{(\chi(1))^2}{|G|}\frac{\phi(n)}{[\mathbb{F}(\zeta):\mathbb{F}(\chi)]}\sum_{g\in G}\frac{\mu(d(g))}{\phi(d(g))}g\bigg)^2 =  \frac{(\chi(1))^2}{|G|}\frac{\phi(n)}{[\mathbb{F}(\zeta):\mathbb{F}(\chi)]}\sum_{g\in G}\frac{\mu(d(g))}{\phi(d(g))}g$,
\end{center}
we get
\begin{center}
  $\frac{(\chi(1))^2}{|G|}\frac{\phi(n)}{[\mathbb{F}(\zeta):\mathbb{F}(\chi)]} = \frac{1}{\sum_{g\in G}\bigg(\frac{\mu(d(g))}{\phi(d(g))}\bigg)^2}.$
\end{center}
Substituting this in (3.1), we get the required expression of $e_\mathbb{F}(\chi)$.
\end{proof}\smallskip

For $\chi\in Irr(G)$, as $Z(\chi)/\text{ker}(\chi)=Z(G/\text{ker}(\chi))$, It is known that $Z(G/\text{ker}(\chi))$ is cyclic and $\chi(1)^2\leq[G:Z(\chi)]$ (see \cite[Corollary 2.30]{Isaacs}). The primitive central idempotent $e_\mathbb{F}(\chi)$ of $\mathbb{F}[G]$ satisfies $\chi(1)^2=[G:Z(\chi)]$, which has already been proved for rational group algebra (see \cite{Bakshi}).
\begin{corollary}\label{C:irreducible}
  Let $\chi$ be an irreducible character corresponding to the representation $\rho$ of the group $G$ thus the degree of the irreducible character $\chi$ is $\sqrt{[G:Z(\chi)]}$. Then
  \begin{center}
    $e_{\mathbb{F}}(\chi)=E_{\text{ker}(\chi),Z(\chi)}$.
  \end{center}\smallskip
  
  Additionally, if $\text{o}(Z(\chi)/\text{ker}(\chi))=p^m$ for some $m\geq1$, then
  \begin{center}
    $e_{\mathbb{F}}(\chi)=\widehat{\text{ker}(\chi)}-\widehat{H}$, 
  \end{center}
where, $H/\text{ker}(\chi)$ is the unique subgroup of $Z(G/\text{ker}(\chi))$ of order $p$.  
\end{corollary}
\begin{proof}
If $\text{ker}(\chi)=G$, then we have $e_{\mathbb{F}}(\chi)=\widehat{G}=E_{G,G}$. If $\text{ker}(\chi)\neq G$, then for $\rho$ the representation affording $\chi$, by Issacs \cite{Isaacs}, $\chi$ annihilates $G/Z(\chi)$, $\rho(g)$ is a scalar matrix $\forall g\in Z(\chi)$ and $Z(\chi)/\text{ker}(\chi)$ is cyclic. Therefore, $\chi$ has the property $\varrho$, and hence by Theorem 3.3, we have
\begin{center}
$e_{\mathbb{F}}(\chi)= \frac{1}{\sum_{g\in Z(\chi)}(\frac{\mu(d(g))}{\phi(d(g))})^{2}}\sum_{g\in Z(\chi)}\frac{\mu(d(g))}{\phi(d(g))}g$.  
\end{center}\smallskip

In case $Z(\chi)=\text{ker}(\chi)$, then the right-hand side of the above equation is equal to $\widehat{\text{ker}(\chi)}=E_{\text{ker}(\chi),\text{ker}(\chi)}$.\smallskip

Suppose $Z(\chi)\neq \text{ker}(\chi)$. Let $\text{o}(Z(\chi)/\text{ker}(\chi))=p^{r_1}_1p^{r_2}_2...p^{r_n}_n$, $p_i's$ distinct primes $\&$ $ r_i's\geq 1$. Let $H/\text{ker}(\chi)$ be a subgroup of $Z(\chi)/\text{ker}(\chi)$ of order $p_1p_2...p_n$. As $\mu(d(g))=0$ for $g\in Z(\chi)/H$, the right-hand side of the above equation in this case is equal to $\frac{1}{\sum_{g\in H}(\frac{\mu(d(g))}{\phi(d(g))})^{2}}\sum_{g\in H}\frac{\mu(d(g))}{\phi(d(g))}g$. Further , we see that $\sum_{g\in H}(\frac{\mu(d(g))}{\phi(d(g))})^{2}$ equals
\begin{center}
 $\sum_{d|p_1p_2...p_n}\frac{\text{o}(\text{ker}(\chi))\phi(d)}{\phi(d)^2}=\text{o}(\text{ker}(\chi))\sum_{d|p_1p_2...p_n}\frac{1}{\phi(d)}=\frac{\text{o}(\text{ker}(\chi))p_1p_2...p_n}{(p_1-1)(p_2-1)...(p_n-1)}$, 
\end{center}
which gives the desired expression for $e_{\mathbb{F}}(\chi)$.\smallskip

Now suppose, in addition, that $\text{o}(Z(\chi)/\text{ker}(\chi))=p^m$, $p$ a prime and $m\geq 1$. Since $d(g)=1$ for $g\in \text{ker}(\chi)$ and $d(g)=P$ for $g\in H/\text{ker}(\chi)$, it follows from above that
\begin{eqnarray}\nonumber
  e_{\mathbb{F}}(\chi) &=& \frac{p-1}{\text{o}(\text{ker}(\chi))p}\bigg(\sum_{g\in \text{ker}(\chi)}g-\frac{1}{p-1}\sum_{g\in H-\text{ker}(\chi)}g\bigg) \\\nonumber
   &=& \frac{p-1}{\text{o}(\text{ker}(\chi))p}\bigg(\frac{p}{p-1}\sum_{g\in \text{ker}(\chi)}g-\frac{1}{p-1}\sum_{g\in H}g\bigg) \\\nonumber
   &=& \frac{p-1}{\text{o}(\text{ker}(\chi))}\sum_{g\in \text{ker}(\chi)}g-\frac{1}{|H|}\sum_{g\in H}g\nonumber,
\end{eqnarray}
and the corollary is proved.
\end{proof}
\begin{corollary}
  Let $\chi\in Irr(G)$ with $\chi(1)^2=[G:Z(G)]$. Then 
  \begin{center}
    $e_{\mathbb{F}}=E_{\text{ker}(\chi),Z(\chi)}$.
  \end{center}
\end{corollary}
\begin{proof}
  Clearly, $Z(G)\subseteq Z(\chi)$. So $\chi(1)^2=[G:Z(\chi)]\geq[G:Z(\chi)]$. But $\chi(1)^2\leq[G:Z(\chi)]$. Therefore, $Z(\chi)=Z(G)$. The result thus follows from corollary 1.
\end{proof}
\begin{corollary}\label{I:idempotent}
If $\chi\in Irr(G)$ be such that $G/Z(\chi)$ is abelian, then $E_{\text{ker},Z(\chi)}$ is the primitive central idempotent of $\mathbb{F}[G]$ associated with $\chi$.
\end{corollary}
\begin{proof}
The assertion follows from \cite{Isaacs} and the preceding corollary.
\end{proof} 
\section{Applications}
In this section, we apply the results from the previous section to explicitly write primitive central idempotents in the finite group algebra of a CM$_{p-1}$-group, a nilpotent group of class $\leq$2. we have give an example, and discuss some consequences of the main results.
\begin{example}
  Let $G$ be a cyclic group of order 3, $G=\langle a\rangle=\{1,a,a^2\}$, and let $\mathbb{F}=\mathbb{F}_{5}$. Consider the complex irreducible characters of $G$:
\begin{center}
  $\chi_{1},\chi_{2},\chi_{3}:G\rightarrow\mathbb{F}(\omega)$, $\omega=e^{2\pi i/3}$
\end{center}  
 where $\chi_{1}$ is the trivial character and $\chi_{2},\chi_{3}$ are the nontrivial characters defined by
 \begin{center}
   $\chi_{2}(a)=\omega, \chi_{3}(a)=\omega^2$.
 \end{center}
 For the nontrivial character $\chi_{2}$, the kernel is trivial i.e. $\text{ker}(\chi_{2})=\{1\}$.\\
 Hence, for each $g\in G$, if we define
 \begin{center}
 $d(g)$ = order of $g$ modulo $\text{ker}(\chi_{2})$ = order of $g$ in $G$, then we have
 \end{center} 
 \begin{center}
   $d(1)=1, d(a)=3$, and $d(a^2)=3$.
 \end{center}
 Hence,
 \begin{center}
 $\sum_{g\in G, \chi(g)\neq 0}\frac{\mu(d(g))}{\phi(d(g))}g = 1-\frac{1}{2}a-\frac{1}{2}a^2$, and $\sum_{g\in G, \chi(g)\neq 0}(\frac{\mu(d(g))}{\phi(d(g))})^2= 1^2+(\frac{-1}{2})^2+(\frac{-1}{2})^2=\frac{3}{2}$,
 \end{center}
 so, we have
 \begin{center}
   $e_{\mathbb{F}}(\chi_{2}) = \frac{1}{3}(2-a-a^2)$.
 \end{center}
 Similarly, for the trivial character $\chi_{1}$ and other nontrivial character $\chi_{3}$, one obtains:
 \begin{center}
   $e_{\mathbb{F}}(\chi_{1}) = \frac{1}{3}(1+a+a^2),  e_{\mathbb{F}}(\chi_{3}) = \frac{1}{3}(2-a^2-a)$. 
 \end{center}
\end{example}
\subsection{CM$_{p-1}$-Groups}\smallskip

Recall that a $p$-group $G$ is called a CM$_{p-1}$-group if every normal subgroup $N$ of $G$ with $Z(G/N)$ cyclic is the kernel of exactly $p-1$ irreducible characters of $G$. It is known that every non principal irreducible character of a CM$_{p-1}$-group $G$ satisfies the hypothesis of Corollary \ref{C:irreducible}. Thus as an immediate consequence of Corollary \ref{C:irreducible}, we have the following theorem.
\begin{theorem}
If $G$ is a CM$_{p-1}$-group, then all primitive central idempotents of $\mathbb{F}[G]$ are given by $\widehat{G}$ and $\widehat{N}-\widehat{H}$, where $N$ runs over all proper normal subgroups of $G$ with $Z(G/N)$ cyclic and for every $H/N$ as the choice of $N$, we have unique subgroup of $Z(G/N)$ of order $p$.
\end{theorem}
\subsection{Nilpotent Group of Class $\leq$2}\smallskip

 For an abelian group $G$, it is known that there is a bijection between the set of primitive central idempotents of $\mathbb{F}[G]$ and the set of normal subgroups $N$ of $G$ with $G/N$ is cyclic. The next theorem extends this result to nilpotent groups of class $\leq$2.
 \begin{theorem}
Let $G$ be a nilpotent group of class $\leq$2. Then there is a bijection between the set of primitive central idempotents of $\mathbb{F}[G]$ and the set of normal subgroup $N$ of $G$ with $Z(G/N)$ is cyclic. Moreover, for a normal subgroup $N$ of $G$ with $Z(G/N)$ cyclic, $E_{N,Z}$ is a primitive central idempotent of $\mathbb{F}[G]$, which is associated with it, where $Z(G)/N=Z(G/N)$.    
 \end{theorem}
 \begin{proof}
Let $e$ be a primitive central idempotent of $\mathbb{F}[G]$. Then there exists an irreducible character $\chi$ of $G$ such that $e = e_{\mathbb{F}}(\chi)$. Let $N =\text{ker($\chi$)}$. Then $Z(G/N)$ is cyclic. We have a mapping $\eta:e=e_{\mathbb{F}}(\chi)\mapsto\text{ker}(\chi)$ from the set of primitive cental idempotents of $\mathbb{F}[G]$ to the set of normal subgroups $N$ of $G$ with $Z(G/N)$ cyclic. For every subgroup $N$ of $G$ with $Z(G/N)$ cyclic $\&$ $G/N$ nilpotent, there exists a faithful irreducible character of $G/N$ and irreducible character $\chi$ of $G$ with $\text{ker}(\chi)=N$, which gives a primitive central idempotent $e_{\mathbb{F}}(\chi)$ corresponding to it. The map $\eta$ is thus onto. Further, since $G$ is nilpotent of class $\leq$2 and $Z(G)\subseteq Z(\chi)\subseteq G$, we have $G/Z(\chi)$ is abelian, so by Corollary \ref{I:idempotent}, $e_{\mathbb{F}}(\chi)=E_{\text{ker},Z(\chi)}$, which depends only on $\text{ker}(\chi)$. Therefore, for $\chi_{1},\chi_{2}\in Irr(G), e_{\mathbb{F}}(\chi_{1})=e_{\mathbb{F}}(\chi_{2})$ if and only if $\text{ker}(\chi_{1})=\text{ker}(\chi_{2})$, which proves that $\eta$ is also one to one.
 \end{proof}
 \begin{proposition}
  If every primitive central idempotent of $\mathbb{F}[G]$ is $*$-clean, then the group algebra $\mathbb{F}[G]$ itself is $*$-clean.
\end{proposition}
\begin{proof}
Let the group algebra $\mathbb{F}[G]$ be semi-simple, and $e_1, e_2, ..., e_m$ be all primitive central idempotents. Then, $1=e_1+e_2+....+e_m$ $\&$ $e_{i}e_{j}=0$ for $i\neq j$. The group algebra $\mathbb{F}[G]$ decomposes as
\begin{equation}
\mathbb{F}[G]=\bigoplus_{i=1}^{m}e_i\mathbb{F}[G]\nonumber
\end{equation}\smallskip

By assumption, each primitive central idempotent $e_i$ is $*$-clean, so in the corner ring $e_i\mathbb{F}[G]$, every element is $*$-clean, or at least the $*$-idempotent $e_i$ itself can be written as $e_i=u_i+f_i$ with $u_i$ a $*$-unit and $f_i$ a $*$-idempotent inside $e_i\mathbb{F}[G]$. Since the decomposition is orthogonal and direct,\smallskip
\begin{center}
  $a=\sum_{i=1}^{n}e_ia$,         $a\in\mathbb{F}[G]$.
\end{center}\smallskip

Each component $e_ia\in e_i\mathbb{F}[G]$ is $*$-clean, because $e_i\mathbb{F}[G]$ is $*$-clean by hypothesis. For each $i$, we write\smallskip
\begin{center}
  $e_ia=u_i+e^{'}_{i}$,
\end{center}\smallskip

where $u_i$ is a $*$-unit in $e_i\mathbb{F}[G]$ and $e^{'}_{i}$ is an $*$-idempotent in $e_i\mathbb{F}[G]$. Then\smallskip
\begin{center}
 $a=\sum_{i=1}^{n}e_ia=\sum_{i=1}^{n}(u_i+e^{'}_{i})=\bigg(\sum_{i=1}^{n}u_i\bigg)+\bigg(\sum_{i=1}^{n}e^{'}_{i}\bigg)$. 
\end{center}\smallskip

The sum $\sum_{i=1}^{n}u_i$ is a $*$-unit in $\mathbb{F}[G]$, because the $u_i$ lies in orthogonal component and is a unit. So $a$ is $*$-clean.\smallskip

Therefore, $\mathbb{F}[G]$ is $*$-clean. 
\end{proof}
\begin{proposition}
Consider the group algebra $\mathbb{F}[G]$ equipped with the standard involution 
\begin{center}
$(\sum_{g\in G}a_gg)^*=\sum_{g\in G}a_gg^{-1}$. 
\end{center}
Then $\mathbb{F}[G]$ is $*$-clean if and only if every primitive central idempotent of $\mathbb{F}[G]$ is a projection (i.e. $e_\mathbb{F}^{*}=e_\mathbb{F}$) with respect to this involution. 
\end{proposition}

\end{document}